\documentclass[12pt,a4paper,reqno]{amsart}
\usepackage{amsmath}
\usepackage{geometry}
\usepackage{amssymb, latexsym}
\usepackage{verbatim}
\usepackage{enumerate}
\usepackage{comment}
\usepackage{txfonts}
\usepackage{mathtools}%                  http://www.ctan.org/pkg/mathtools
\usepackage[tableposition=top]{caption}% http://www.ctan.org/pkg/caption
\usepackage{booktabs,dcolumn}%           http://www.ctan.org/pkg/dcolumn + http://www.ctan.org/pkg/booktabs

\newtheorem{theorem}{Theorem}[section]
\newtheorem{proposition}[theorem]{Proposition}

\newtheorem{lemma}[theorem]{Lemma}
\newtheorem{corollary}[theorem]{Corollary}
\newtheorem{conjecture}[theorem]{Conjecture}

\newtheorem{remark}[theorem]{Remark}

\newcommand\R{\mathbb{R}}
\newcommand\Z{\mathbb{Z}}

\newcommand\eps{\varepsilon}

\begin{document}
%\begin{frontmatter}
\title[Inverse Kneser inequality]{An inverse theorem for an inequality of Kneser}

\author{Terence Tao}
\address{Department of Mathematics, UCLA\\
405 Hilgard Ave\\
Los Angeles CA 90095\\
USA}
\email{tao@math.ucla.edu}

%\author{Joni Ter\"av\"ainen}
%\address{Department of Mathematics and Statistics, University of Turku\\
%20014 Turku\\
%Finland}
%\email{joni.p.teravainen@utu.fi}

\begin{abstract}	Let $G = (G,+)$ be a compact connected abelian group, and let $\mu_G$ denote its probability Haar measure.  A theorem of Kneser (generalising previous results of Macbeath, Raikov, and Shields) establishes the bound
$$ \mu_G(A + B) \geq \min( \mu_G(A)+\mu_G(B), 1 ) $$
whenever $A,B$ are compact subsets of $G$, and $A+B \coloneqq \{ a+b: a \in A, b \in B \}$ denotes the sumset of $A$ and $B$.  Clearly one has equality when $\mu_G(A)+\mu_G(B) \geq 1$.  Another way in which equality can be obtained is when $A = \phi^{-1}(I), B = \phi^{-1}(J)$ for some continuous surjective homomorphism $\phi: G \to \R/\Z$ and compact arcs $I,J \subset \R/\Z$.  We establish an inverse theorem that asserts, roughly speaking, that when equality in the above bound is almost attained, then $A,B$ are close to one of the above examples.  We also give a more ``robust'' form of this theorem in which the sumset $A+B$ is replaced by the partial sumset $A +_\eps B \coloneqq \{ 1_A * 1_B \geq \eps \}$ for some small $\eps>0$.  In a subsequent paper with Joni Ter\"av\"ainen, we will apply this latter inverse theorem to establish that certain patterns in multiplicative functions occur with positive density.
\end{abstract}

\maketitle
%\end{frontmatter}
%%%%%%%%%%%%%%%%%%%%%%%%%

\section{Introduction}

Throughout this paper, we use $\mu_G$ to denote the Haar probability measure on any compact abelian group $G = (G,+)$; thus for instance $\mu_{\R/\Z}$ is Lebesgue measure on the unit circle $\R/\Z$.
In \cite{kneser}, Kneser established\footnote{In a previous version of this manuscript, this inequality was incorrectly attributed to Kemperman.  We thank John Griesmer for pointing out this error.} the inequality
\begin{equation}\label{kemp}
 \mu_G(A + B) \geq \min( \mu_G(A)+\mu_G(B), 1 ) 
\end{equation}
whenever $A,B$ are non-empty compact subsets of a compact connected abelian group $G$, and $A+B \coloneqq \{ a+b: a \in A, b \in B \}$ denotes the sumset of $A$ and $B$.  A subsequent result of Kemperman \cite{kemperman} extended this inequality to compact connected nonabelian groups also, but we restrict attention here to the abelian case.  Prior to Macbeath's result, the case of a circle $G = \R/\Z$ was obtained by Raikov \cite{raikov}
(and can also be derived by a limiting argument from the Cauchy-Davenport inequality), the case of a torus $G = (\R/\Z)^d$ was obtained by Macbeath \cite{macbeath}, and the case of second countable connected compact groups by Shields \cite{shields}.  The fact that $G$ is connected is crucial, since otherwise $G$ could contain open subgroups of measure strictly between $0$ and $1$, which would of course yield a counterexample to \eqref{kemp}.

In a blog post \cite{blog} of the author, it was observed that one could use an argument of Ruzsa \cite{ruzsa} to obtain the following stronger bound (cf. Pollard's bound \cite{pollard} for cyclic groups):

\begin{theorem}\label{ruzsa-thm}  Let $A,B$ be measurable subsets of a compact connected abelian group $G$.  Then
$$ \int_G \min( 1_A * 1_B, t )\ d\mu_G \geq t \min( \mu_G(A) + \mu_G(B) -t, 1 ) $$
for any $0 \leq t \leq \min(\mu_G(A), \mu_G(B))$, where 
$$ 1_A * 1_B(x) \coloneqq \int_G 1_A(y) 1_B(x-y)\ d\mu_G(y)$$
is the convolution of $1_A$ and $1_B$, and $1_A$ denotes the indicator function of $A$.
\end{theorem}

For the convenience of the reader, we give the proof of this theorem in Section \ref{ruz}.  To see why this result implies \eqref{kemp}, we observe the following corollary of Theorem \ref{ruzsa-thm}.  Given two measurable subsets $A,B$ of $G$ and a parameter $\eps>0$, we define the partial sumset $A +_\eps B$ by the formula
$$ A +_\eps B := \{ x \in G: 1_A * 1_B(x) \geq \eps \}.$$
This is a compact subset of $A+B$.

\begin{corollary}\label{kemp-cor}  Let $G, A, B$ be as in Theorem \ref{ruzsa-thm}.  Then for any $0 < \eps < \min(\mu_G(A),\mu_G(B))^2$, we have
$$ \mu_G(A +_\eps B) \geq \min( \mu_G(A) + \mu_G(B), 1 ) - 2 \sqrt{\eps} $$
\end{corollary}

One can improve the error term $2\sqrt{\eps}$ slightly, but we will not need to do so here.

\begin{proof}  From the pointwise bound
$$ \min( 1_A * 1_B, \sqrt{\eps} ) \leq \eps + \sqrt{\eps} 1_{A +_\eps B}$$
one has
$$ \int_G \min( 1_A * 1_B, \sqrt{\eps} )\ d\mu_G \leq \eps + \sqrt{\eps}  \mu_G( A +_\eps B)$$
and hence by Theorem \ref{ruzsa-thm}, we have
$$ \mu_G(A +_\eps B)  \geq \min( \mu_G(A) + \mu_G(B) - \sqrt{\eps}, 1 )  - \sqrt{\eps},$$
giving the claim.
\end{proof}

Since the set $A +_\eps B$ is contained in $A+B$, the claim \eqref{kemp} follows from this corollary (in the case $\mu_G(A), \mu_G(B) > 0$) by sending $\eps$ to $0$, noting that \eqref{kemp} is trivial when $\mu_G(A)=0$ or $\mu_G(B)=0$.

There are several cases in which the estimate \eqref{kemp} is sharp.  Firstly, one has the trivial cases in which $A$ or $B$ is a point; there are some further examples of this type where (say) $A$ is a coset of a measure zero subgroup of $G$, and $B$ is a union of cosets of that group.  Secondly, if one has $\mu_G(A) + \mu_G(B) \geq 1$, then the compact sets $A$ and $x-B$ cannot be disjoint (as this would disconnect $G$, since the complement of $A \cup (x-B)$ would be an open null set and hence empty); hence $A+B=G$ and \eqref{kemp} holds with equality.  
Define a \emph{Bohr set} to be a subset of $G$ of the form $\phi^{-1}(I)$, where $\phi: G \to \R/\Z$ is a continuous surjective homomorphism and $I$ is a compact arc in $\R/\Z$ (i.e., a set of the form $I = [a,b] \text{ mod } \Z$ for some $a < b$, where $x \mapsto x \text{ mod } \Z$ is the projection from $\R$ to $\R/\Z$), and say that two Bohr sets $\phi^{-1}(I), \psi^{-1}(J)$ are \emph{parallel} if $\phi=\psi$.  If $A = \phi^{-1}(I)$ and $B = \phi^{-1}(J)$ are two parallel Bohr sets, then $A+B = \phi^{-1}(I+J)$ is also a Bohr set, and (by the uniqueness of Haar measure) the Haar measure of $A,B,A+B$ is equal to the measures of $I,J,I+J$ respectively on the unit circle. One can then easily verify that \eqref{kemp} holds with equality in these cases.

The main result of this paper is an inverse theorem that asserts, roughly speaking, that the above examples are essentially the only situations in which equality can occur.  More precisely, we have

\begin{theorem}[Inverse theorem, first form]\label{inv-1}  Let $\eps>0$, and suppose that $\delta>0$ is sufficiently small depending on $\eps$.  Then, for any compact subsets $A,B$ of a compact connected abelian group $G = (G,+)$ with
$$ \mu_G(A), \mu_G(B), 1 - \mu_G(A) - \mu_G(B) \geq \eps$$
and
$$ \mu_G(A+B) \leq \mu_G(A) + \mu_G(B) + \delta,$$
there exist parallel Bohr sets $\phi^{-1}(I), \phi^{-1}(J)$ such that
$$ \mu_G( A \Delta \phi^{-1}(I) ), \mu_G( B \Delta \phi^{-1}(J) ) \leq \eps,$$
where $A \Delta B$ denotes the symmetric difference of $A$ and $B$.
\end{theorem}

In the case $G = \R/\Z$, this result was recently obtained in \cite[Theorem 1.5]{candela} (with a quite sharp dependence between $\eps$ and $\delta$), by a different method; see also the earlier work \cite{mfy}, \cite{fjm}.  In the case of a torus $G = (\R/\Z)^d$, when the measures of $A$ and $B$ are small and comparable to each other, this theorem was obtained (again with a sharp dependence between $\eps$ and $\delta$) in \cite[Theorem 1.4]{bilu}.

As a consequence of the above theorem, we can reprove a theorem of Kneser \cite[Satz 2]{kneser} classifying when equality holds in \eqref{kemp}:

\begin{corollary}\label{cor}  Let let $A,B$ be non-empty compact subsets of a compact connected abelian group $G$ such that equality holds in \eqref{kemp}.  Then at least one of the following statements hold:
\begin{itemize}
\item[(i)]  $\mu_G(A)=0$ or $\mu_G(B) = 0$.
\item[(ii)]  $A,B$ are parallel Bohr sets.
\item[(iii)] $\mu_G(A) + \mu_G(B) \geq 1$.
\end{itemize}
\end{corollary}

We prove this corollary in Section \ref{cor-sec}.  

Much as \eqref{kemp} can be deduced from Corollary \ref{kemp-cor}, Theorem \ref{inv-1} will be deduced from the following variant:

\begin{theorem}[Inverse theorem, second form]\label{inv-2}  Let $\eps>0$, and suppose that $\delta>0$ is sufficiently small depending on $\eps$.  Then, for any measurable subsets $A,B$ of a compact connected abelian group $G$ with
$$ \mu_G(A), \mu_G(B), 1 - \mu_G(A) - \mu_G(B) \geq \eps$$
and
$$ \mu_G(A +_\delta B) \leq \mu_G(A) + \mu_G(B) + \delta,$$
there exist parallel Bohr sets $\phi^{-1}(I), \phi^{-1}(J)$ such that
$$ \mu_G( A \Delta \phi^{-1}(I) ), \mu_G( B \Delta \phi^{-1}(J) ) \leq \eps.$$
\end{theorem}

Since $A +_\delta B$ is clearly contained in $A+B$, it is immediate that Theorem \ref{inv-2} implies Theorem \ref{inv-1}.  

The proof of Theorem \ref{inv-2} can be outlined as follows.  To simplify this outline, let us ignore all the $\eps$ and $\delta$ errors, in particular pretending that the partial sumset $A +_\delta B$ is the same as the full sumset $A+B$.  Let us informally call a pair $(A,B)$ a ``critical pair'' if the conditions of Theorem \ref{inv-2} are obeyed.  By using ``submodularity inequalities'' such as
$$ \mu_G( (A_1 \cup A_2) + B) + \mu_G( (A_1 \cap A_2) + B) \leq \mu_G(A_1+B) + \mu_G(A_2+B),$$
valid for any compact $A_1,A_2,B \subset G$, (which follow from the identity $(A_1 \cup A_2)+B = (A_1+B) \cup (A_2+B)$ and the inclusion $(A_1 \cap A_2)+B \subset (A_1+B) \cap (A_2+B)$ respectively), one can obtain a number of closure properties regarding critical pairs, for instance establishing that if $(A_1,B)$ and $(A_2,B)$ are critical pairs then $(A_1 \cup A_2,B)$ and $(A_1 \cap A_2,B)$ are also, provided that $A_1 \cap A_2$ is non-empty and $A_1 \cup A_2$ is not too large.  Similarly, using the associativity $(A+B)+C = A+(B+C)$ of the sum set operation, one can show that if $(A,B)$ and $(A+B,C)$ are critical pairs, then so are $(B,C)$ and $(A,B+C)$.  Using such closure properties repeatedly in combination with the translation invariance of the critical pair concept, we can start with a critical pair $(A,B)$ and generate a small (but non-trivial) auxiliary set $C$ such that $(A,C)$ and $(C,C)$ are critical pairs; furthermore, we can also arrange matters so that $(C,kC)$ is a critical pair for all bounded $k$ (e.g. all $1 \leq k \leq 10^4$), where $kC = C + \dots + C$ is the $k$-fold iterated sumset of $C$.  This implies in particular that $C$ has linear growth in the sense that $\mu_G(kC) \approx k\mu_G(C)$ for all bounded $k$, which by existing tools in inverse sumset theory (in particular using arguments of Schoen \cite{schoen} and Green-Ruzsa \cite{rect}, \cite{green}) can be used to show that $C$ is very close to a Bohr set.  As $(A,C)$ is a critical pair, some elementary analysis can then be deployed to show that $A$ is very close to a Bohr set parallel to $C$, and then as $(A,B)$ is also critical, $B$ is also very close to a Bohr set parallel to $A$, giving the claim.

In order to make notions such as ``critical pair'' rigorous, it will be convenient to use the language of ``cheap nonstandard analysis'' \cite{cheap}, working with a sequence $(A,B) = (A_n,B_n)$ of pairs in a sequence $G = G_n$ of groups, rather than with a single pair in a single group, so that asymptotic notation such as $o(1)$ can be usefully deployed.  It should however be possible to reformulate the arguments below without this language, at the cost of having to pay significantly more attention to various $\eps$ and $\delta$ type parameters.

In a subsequent paper with Joni Ter\"av\"ainen, we will combine this theorem with the structural theory of correlations of bounded multiplicative functions (as developed recently in \cite{jt}) to obtain new results about the distribution of sign patterns $(f_1(n+1), f_2(n+2),\dots, f_k(n+k))$ of various bounded multiplicative functions $f_1,\dots,f_k$ such as the Liouville function $\lambda(n)$, as well as generalisations such as $e^{2\pi i \Omega(n)/m}$ for a fixed natural number $m$, where $\Omega(n)$ denotes the number of prime factors of $n$ (counting multiplicity).

\begin{remark}  Results analogous to Theorem \ref{inv-1} are known when the connected group $G$ is replaced by the discrete group $\Z/p\Z$: see \cite{freiman}, \cite{rodseth}, \cite{serra}, \cite{rect}, \cite{blr}, \cite{g}, as well as some further discussion in \cite{hgz}.  In the recent paper \cite{candela}, these results (particularly those in \cite{g}) are used to establish the $G=\R/\Z$ case of Theorem \ref{inv-1}.  On the integers $\Z$, a version of Theorem \ref{inv-2} when $A, B \subset \Z$ have the same cardinality was obtained very recently in \cite[Corollary 5.2]{shao}.
\end{remark}

\subsection{Acknowledgments}

The author was supported by a Simons Investigator grant, the James and Carol Collins Chair, the Mathematical Analysis \&
Application Research Fund Endowment, and by NSF grant DMS-1266164.  The author is indebted to Joni Ter\"av\"ainen for key discussions that led to the author pursuing this question, and for helpful comments and corrections, and to Ben Green for some references.  The author also thanks John Griesmer and the anonymous referees for further corrections and suggestions.

\section{Proof of Theorem \ref{ruzsa-thm}}\label{ruz}

We now prove Theorem \ref{ruzsa-thm}.  By inner regularity of Haar measure and a limiting argument we may assume $A,B$ are compact.  In the case
$$ \mu_G(A) + \mu_G(B) - t \geq 1,$$
we see that the set $A \cap (x-B) = \{ y \in A: x-y \in B \}$ has measure at least $\mu_G(A)+\mu_G(B)-1 \geq t$ for every $x \in G$, and hence $1_A * 1_B(x) \geq t$ for all $x \in G$, giving the claim in this case.  Thus we may assume that $\mu_G(A)+\mu_G(B)-t < 1$.  We may also assume that $G$ is non-trivial, which (by the connectedness of $G$) implies that there exist measurable subsets of $G$ of arbitrary measure between $0$ and $1$.

Fix $G$, let $B$ be a compact subset of $G$, and let $0 \leq t \leq \mu_G(B)$ be a real number.  For any compact $A \subset G$, define the quantity
$$ c(A) \coloneqq \int_G \min(1_A * 1_B, t)\ d\mu_G - t (\mu_G(A) + \mu_G(B) -t),$$
and then for every $a \in [0,1]$, let $f(a)$ denote the infimum of $c(A)$ over all $A$ with $\mu_G(A)=a$.  Our task is to show that $f$ is non-negative on the interval $[t, 1-\mu_G(B)+t]$.

If $\mu_G(A) = 1-\mu_G(B)+t$, then by the previous discussion we have $1_A * 1_B(x) \geq t$ for all $x \in G$, and hence $c(A)=0$; hence $f(1-\mu_G(B)+t)=0$.  At the other extreme, if $\mu_G(A) = t$, then $1_A * 1_B(x) \leq t$ for all $x \in G$, and hence from Fubini's theorem we again have $c(A) = 0$.

Observe that if one modifies $A$ by a set of measure at most $\delta$, then $c(A)$ varies by $O(\delta)$.  From this we conclude that $f$ is Lipschitz continuous.  Thus, if we assume for contradiction that $f$ is not always non-negative; then there must exist a point $a$ in the interior of $[t, 1-\mu_G(B)+t]$ where $f$ attains a global negative minimum and is not locally constant in a neighbourhood of $a$.  In particular, there exist arbitrarily small $\eps$ such that
\begin{equation}\label{fae}
 f(a) < \frac{f(a-\eps) + f(a+\eps)}{2}.
\end{equation}

On the other hand, we observe the crucial submodularity property
\begin{equation}\label{submod}
 c(A_1) + c(A_2) \geq c(A_1 \cap A_2)+ c(A_1 \cup A_2)
\end{equation}
for all measurable sets $A_1,A_2 \subset G$.  To see this, we begin with the inclusion-exclusion identity
$$ 1_{A_1} + 1_{A_2} = 1_{A_1 \cap A_2} + 1_{A_1 \cup A_2}$$
which implies that
$$
 1_{A_1} * 1_B + 1_{A_2} * 1_B = 1_{A_1 \cap A_2} * 1_B + 1_{A_1 \cup A_2} * 1_B.
$$
Observe that for each $x \in G$, we have the pointwise inequalities
$$1_{A_1 \cap A_2} * 1_B(x) \leq 1_{A_1} * 1_B(x), 1_{A_2} * 1_B(x) \leq 1_{A_1 \cup A_2} * 1_B(x);$$
by the concavity of the map $x \mapsto \min(x,t)$ we therefore have the pointwise bound
\begin{equation}\label{a12b}
 \min( 1_{A_1} * 1_B , t ) + \min( 1_{A_2} * 1_B, t) \geq \min( 1_{A_1 \cap A_2} * 1_B, t) + \min( 1_{A_1 \cup A_2} * 1_B, t).
\end{equation}
Integrating over $G$ and using the inclusion-exclusion formula $\mu_G(A_1) + \mu_G(A_2) = \mu_G(A_1 \cap A_2) + \mu_G(A_1 \cup A_2)$, we obtain \eqref{submod} as desired.

Let $A$ be such that $\mu_G(A)=a$, and let $\eps>0$ be a small quantity such that \eqref{fae} holds.  Now we observe the following application of connectedness:

\begin{lemma}\label{cont}  Let $A$ be a measurable subset of $G$, and let $t$ be any real number with $\mu_G(A)^2 \leq t \leq \mu_G(A)$.  Then there exists $x \in G$ such that $\mu_G(A \cap (x+A)) = t$.
\end{lemma}

\begin{proof}
 The function $x \mapsto 1_A * 1_{-A}(x) = \mu_G( A \cap (x+A))$, being a convolution of $L^2$ functions, is a continuous function of $x$ that equals $\mu_G(A)$ when $x=0$, and has a mean value of $\mu_G(A)^2$ on $G$ by Fubini's theorem.  The claim then follows from the intermediate value theorem and the connectedness of $G$.
\end{proof}

By Lemma \ref{cont}, there exists $x \in G$ such that $\mu_G( A \cap (x+A) ) = a - \eps$, and hence by inclusion-exclusion $\mu_G(A \cup (x+A) ) = a+\eps$.  From \eqref{submod} with $A_1,A_2$ replaced by $A, x+A$ we have
$$ c( A) + c(x+A) \geq c(A \cap (x+A)) + c(A \cup (x+A)) \geq f(a-\eps) + f(a+\eps).$$
By translation invariance we have $c(x+A) = c(A)$, hence
$$ 2c(A) \geq f(a-\eps) + f(a+\eps).$$
Taking infima over all $A$ with $\mu_G(A) = a$, we contradict \eqref{fae}, and the claim follows.

\begin{remark}  With some minor notational modifications, this argument also works for nonabelian compact connected groups; see \cite{blog}.
\end{remark}

\section{Proof of Corollary \ref{cor}}\label{cor-sec}

We now prove Corollary \ref{cor}.  Suppose that $A,B$ are compact subsets of a compact connected abelian group $G$ are such that equality holds in \eqref{kemp}.  We may assume that $\mu_G(A), \mu_G(B), 1 - \mu_G(A)-\mu_G(B) > 0$, since we are done otherwise.  Applying Theorem \ref{inv-1}, we conclude that there exist sequences $\phi_n^{-1}(I_n), \phi_n^{-1}(J_n)$ of parallel Bohr sets such that
$$ \mu_G( A \Delta \phi^{-1}_n(I_n) ), \mu_G( B \Delta \phi^{-1}_n(J_n) )  = o(1),$$
where in this section we use $o(1)$ to denote a quantity that goes to zero as $n \to \infty$.  In particular, the arcs $I_n,J_n$ in the circle $\R/\Z$ have measure
$$ \mu_{\R/\Z}(I_n) = \mu_G(A) + o(1), \mu_{\R/\Z}(J_n) = \mu_G(B) + o(1).$$
Taking Fourier coefficients, we see that
\begin{align*}
 \left| \int_G 1_A(x) e^{2\pi i \phi_n(x)}\ d\mu_G(x) \right| &= \left|\int_{I_n} e^{2\pi i \alpha}\ d\mu_{\R/\Z}(\alpha)\right| + o(1) \\
&= \frac{1}{\pi} \sin(\pi \mu_{\R/\Z}(I_n)) + o(1) \\
&= \frac{1}{\pi} \sin(\pi \mu_G(A)) + o(1).
\end{align*}
On the other hand, from Plancherel's theorem we have
$$ \sum_{\phi \in \hat G}  \left| \int_G 1_A(x) e^{2\pi i \phi(x)}\ d\mu_G(x) \right|^2 = \mu_G(A)$$
where the Pontryagin dual group $\hat G$ consists of all continuous homomorphisms $\phi$ from $G$ to $\R/\Z$.  Thus, for $n$ large enough, there are only boundedly many possible choices for $\phi_n$, and by passing to a subsequence if necessary we may assume that $\phi_n = \phi$ does not depend on $n$.  For $n,n' \to \infty$, we now have
$$ \mu_G( A \Delta \phi^{-1}(I_n) ), \mu_G( A \Delta \phi^{-1}(I_{n'}) ) \to 0,$$
and hence by the triangle inequality
$$ \mu_{\R/\Z}( I_n \Delta I_{n'} ) = \mu_G( \phi^{-1}(I_n) \Delta \phi^{-1}(I_{n'}) ) \to 0$$
as $n,n' \to \infty$.
By the Bolzano-Weierstrass theorem, we may thus find a compact arc $I$ independent of $n$ such that
$$ \mu_{\R/\Z}( I_n \Delta I) \to 0$$
as $n \to \infty$, which implies that
$$ \mu_G( \phi^{-1}(I_n) \Delta \phi^{-1}(I) ) \to 0$$
as $n \to \infty$.  Hence by the triangle inequality, $A$ and $\phi^{-1}(I)$ must agree $\mu_G$-almost everywhere; as $A$ is compact, it cannot omit any interior point of $\phi^{-1}(I)$ (as this would also exclude a set of positive $\mu_G$ measure from $A$, and hence $A$ must therefore consist of the union of $\phi^{-1}(I)$ and a $\mu_G$-null set $E$.  Similarly, there is a compact arc $J$ such that $B$ consists of the union of $\phi^{-1}(J)$ and a $\mu_G$-null set $F$.  Thus $A+B$ contains $\phi^{-1}(I+J)$, which has measure 
$$\mu_{\R/\Z}(I+J) = \mu_{\R/\Z}(I)+\mu_{\R/\Z}(J) = \mu_G(\phi^{-1}(I)) + \mu_G(\phi^{-1}(J)) = \mu_G(A) + \mu_G(B);$$
since \eqref{kemp} holds, we conclude that $A+B$ is in fact equal to the union of $\phi^{-1}(I+J)$ and a $\mu_G$-null set.  Thus for every $a \in A$, the set $a + \phi^{-1}(J)$ lies in the union of $\phi^{-1}(I+J)$ and a $\mu_G$-null set, which forces $a$ to lie in $\phi^{-1}(I)$; thus $A = \phi^{-1}(I)$, and similarly $B = \phi^{-1}(J)$, giving the claim.

\section{Proof of Theorem \ref{inv-2}}\label{cheap}

We now prove Theorem \ref{inv-2}.  It will be convenient to reformulate the result in terms of a ``cheap'' form of nonstandard analysis (as used in \cite{cheap}), involving sequences of potential counterexamples.  The full machinery of nonstandard analysis, such as ultraproducts and the construction of Loeb measure, will not be needed for this reformulation; one could certainly insert such machinery into the arguments below, but they do not appear to dramatically simplify the proofs.

We will need a natural number parameter $n$.  In the sequel, all mathematical objects will be permitted to depend on this parameter (and can thus be viewed as a sequence of objects), unless explicitly declared to be ``fixed''.  Usually we will suppress the dependence on $n$.  For instance, a sequence $G_n$ of compact abelian groups will be abbreviated as $G = G_n$.  A real number $x = x_n$ depending on $n$ is said to be \emph{infinitesimal} if one has $\lim_{n \to \infty} x_n = 0$, in which case we write $x = o(1)$.  If $x = x_n$, $y = y_n$ are real numbers such that $|x_n| \leq Cy_n$ for all sufficiently large $n$ and some fixed $C>0$, we write $x \ll y$, $y \gg x$, or $x = O(y)$.
Two measurable subsets $A = A_n$, $B = B_n$ of a compact abelian group $G = G_n$ are said to be \emph{asymptotically equivalent} if one has $\mu_G( A \Delta B ) = o(1)$.  This is clearly an equivalence relation.

Theorem \ref{inv-2} can now be deduced from the following variant:

\begin{theorem}[Inverse theorem, cheap nonstandard form]\label{inv-3}  Let $A = A_n, B = B_n$ be measurable subsets of a sequence $G = G_n$ of compact connected abelian groups with
\begin{equation}\label{muab}
 \mu_G(A), \mu_G(B), 1 - \mu_G(A) - \mu_G(B) \gg 1
\end{equation}
and
\begin{equation}\label{muab-2}
 \mu_G(A +_\delta B) \leq \mu_G(A) + \mu_G(B) + o(1)
\end{equation}
for some infinitesimal $\delta > 0$. Then there exist parallel Bohr sets $\phi^{-1}(I) = \phi_n^{-1}(I_n)$ and $\phi^{-1}(J) = \phi_n^{-1}(J_n)$ in $G = G_n$ such that $A$ and $B$ are asymptotically equivalent to $\phi^{-1}(I), \phi^{-1}(J)$ respectively.
\end{theorem}

Let us assume Theorem \ref{inv-3} for now and see how it implies Theorem \ref{inv-2} (and hence also Theorem \ref{inv-1}).  Suppose for contradiction that Theorem \ref{inv-2} fails.  Carefully negating the quantifiers, and applying the axiom of choice, we conclude that there exists an $\eps>0$, such that for every natural number $n$ there are measurable subsets $A = A_n, B = B_n$ of a compact connected abelian group $G = G_n$ such that for every $n$ one has
$$ \mu_G(A), \mu_G(B), 1 - \mu_G(A) - \mu_G(B) \geq \eps$$
and
$$ \mu_G(A +_{1/n} B) \leq \mu_G(A) + \mu_G(B) + \frac{1}{n},$$
but such that for each $n$, there do \emph{not} exist parallel Bohr sets $\phi_n^{-1}(I_n), \phi_n^{-1}(J_n)$ such that
$$ \mu_G( A_n \Delta \phi_n^{-1}(I_n) ), \mu_G( B_n \Delta \phi^{-1}(J_n) ) \leq \eps.$$
By applying Theorem \ref{inv-3} with the infinitesimal $\delta = \delta_n \coloneqq \frac{1}{n}$, we know that $A,B$ are asymptotically equivalent respectively to parallel Bohr sets $\phi^{-1}(I), \phi^{-1}(J)$.  But by taking $n$ large enough, this contradicts the previous statement.

It remains to prove Theorem \ref{inv-3}.  One of the main reasons of passing to this formulation is that it allows for\footnote{The price one pays for this is that it is difficult to directly extract from this argument an explicit dependence of $\delta$ on $\eps$ in Theorem \ref{inv-2}.  However, this can be done (in principle, at least) by refraining from passing to the ``cheap nonstandard'' framework and instead working with a more quantitative, but significantly messier, notion of critical pair, in which one replaces all $o(1)$ errors by more explicit decay rates that may vary from line to line.  We leave this task to the interested reader.} the following convenient definition.  In the sequel $G = G_n$ is understood to be a sequence of compact connected abelian groups with probability Haar measure $\mu = \mu_n$.  A pair $(A,B)$ of measurable subsets of $G$ is said to be a \emph{critical pair}\footnote{A more accurate terminology would be ``asymptotically critical pair'', but we use ``critical pair'' instead for brevity.}  if one has the properties \eqref{muab}, \eqref{muab-2} for some infinitesimal $\delta>0$.  Our goal is thus to prove that every critical pair is equivalent to a pair of parallel Bohr sets.

It turns out that the space of critical pairs is closed under a number of operations.  Clearly it is symmetric: $(A,B)$ is a critical pair if and only if $(B,A)$ is.  It is also obvious that if $(A,B)$ is a critical pair, then so is $(A+x,B+y)$ for any $x,y \in G$, where $A+x \coloneqq \{ a +x: a \in A \}$ denotes the translate of $A$ by $x$.
Next, we observe that it is insensitive to asymptotic equivalence:

\begin{lemma}\label{crit-equiv}  Suppose that $(A,B)$ is a critical pair, and that $A'$ is asymptotically equivalent to $A$.  Then $(A',B)$ is also a critical pair.
\end{lemma}

Of course by symmetry, the same statement holds if we replace $B$ by an asymptotically equivalent $B'$.  Thus one only needs to know $A,B$ up to asymptotic equivalence to determine if $(A,B)$ form a critical pair.

\begin{proof}  By hypothesis, there exists an infinitesimal $\eps>0$ such that
$$ \mu_G(A' \Delta A) \leq \eps,$$
which implies the pointwise bound
$$ | 1_{A'} * 1_B - 1_A * 1_B | \leq \eps$$
and hence we have the inclusion
$$ A' +_{\delta+\eps} B \subset A +_\delta B$$
for any $\delta>0$.  On the other hand, as $(A,B)$ is a critical pair, there exists an infinitesimal $\delta>0$ such that
$$  \mu_G(A +_\delta B) \leq \mu_G(A) + \mu_G(B) + o(1),$$
and hence
$$  \mu_G(A' +_{\delta+\eps} B) \leq \mu_G(A') + \mu_G(B) + o(1).$$
From this we easily verify that $(A',B)$ is a critical pair as claimed.
\end{proof}

We can now simplify the problem by observing that if one element $(A,B)$ of a critical pair is already asymptotically equivalent to a Bohr set, then so is the other:

\begin{proposition}\label{p1}  Let $(A,B)$ be a critical pair, and suppose that $B$ is asymptotically equivalent to a Bohr set $\phi^{-1}(J)$. Then $A$ is asymptotically equivalent to a parallel Bohr set $\phi^{-1}(I)$.
\end{proposition}

\begin{proof}  By Lemma \ref{crit-equiv}, we may assume without loss of generality that $B = \phi^{-1}(J)$; also, by translation invariance we may assume that $J = [0,t] \text{ mod } \Z$ for some $t$ with
\begin{equation}\label{muat}
\mu_G(A), t, 1 - \mu_G(A) - t \gg 1.
\end{equation}
As $(A,B)$ is a critical pair, there exists an infinitesimal $\delta>0$ such that the set $C \coloneqq A +_\delta B$ has measure
$$\mu_G(C) = \mu_G(A) + \mu_G(B) + o(1) = \mu_G(A) + t + o(1).$$

The set $B$ is invariant with respect to translations in the kernel of $\phi$, so $C$ is similarly invariant, thus $C = \phi^{-1}(E)$ for some measurable subset $E$ of $\R/\Z$ with 
\begin{equation}\label{met}
\mu_{\R/\Z}(E) = \mu_G(C) = \mu_G(A) + t + o(1).
\end{equation}

The pullback map $\phi^*: g \mapsto g \circ \phi$ is an isometry from $L^2(\R/\Z, \mu_{\R/\Z})$ to $L^2(G, \mu_G)$.  Taking adjoints, we obtain a pushforward map $\phi_*: L^2( G, \mu_G ) \mapsto L^2(\R/\Z, \mu_{\R/\Z})$ such that
$$ \int_{\R/\Z} \phi_*(f)(\alpha) g(\alpha)\ d\mu_{\R/\Z}(\alpha) = \int_G f(x) g(\phi(x))\ d\mu_G(x)$$
for all $f \in L^2(G,\mu_G)$ and $g \in L^2(\R/\Z, \mu_{\R/\Z})$.  It is easy to see that the map $\phi_*$ is monotone with $\phi_*(1)=1$ (up to almost everywhere equivalence).  If we write $f_A \coloneqq \phi_* 1_A$ for the pushforward of $1_A$, then $f_A$ takes values in $[0,1]$ (after modifying on a set of measure zero if necessary), and we have
\begin{equation}\label{rza}
 \int_{\R/\Z} f_A\ d\mu_{\R/\Z} = \int_G 1_A\ d\mu_G = \mu_G(A).
\end{equation}
Also, since $1_A * 1_B = 1_A * 1_{\phi^{-1}(J)}$ is bounded by $o(1)$ outside of $C = \phi^{-1}(E)$, we see that $f_A * 1_J$ is bounded almost everywhere by $o(1)$ outside of $E$, thus
\begin{equation}\label{rze}
 \int_{(\R/\Z) \backslash E} f_A * 1_{[0,t] \text{ mod } \Z}\ d\mu_{\R/\Z} = o(1).
\end{equation}
Let $\lambda > 0$ be any fixed parameter, and let $F_\lambda \subset \R/\Z$ denote the set $F_\lambda \coloneqq \{ f_A \geq \lambda \}$, then we have
$$ \int_{(\R/\Z) \backslash E} 1_{F_\lambda} * 1_{[0,t] \text{ mod } \Z}\ d\mu_{\R/\Z} = o(1).$$
From Markov's inequality, we conclude that for any fixed $\eps>0$, all but $o(1)$ in measure of the set $F_\lambda +_\eps J$ is contained in $E$, thus
$$ \mu_{\R/\Z}( F_\lambda +_\eps J ) \leq \mu_{\R/\Z}(E) + o(1) = \mu_G(A) + t + o(1).$$
On the other hand, from Corollary \ref{kemp-cor} we have
$$ \mu_{\R/\Z}( F_\lambda +_\eps J ) \geq \min( \mu_{\R/\Z}(F_\lambda) + t, 1 ) - 2 \sqrt{\eps};$$
combining the two bounds and sending $\eps$ to zero, we conclude using \eqref{muat} that
$$ \mu_{\R/\Z}(F_\lambda) \leq \mu_G(A) + o(1).$$
for any fixed $\lambda>0$.  Sending $\lambda$ sufficiently slowly to zero as $n \to \infty$, we conclude on diagonalising that
$$ \mu_{\R/\Z}( F_\kappa ) \leq \mu_G(A) + o(1)$$
for some infinitesimal $\kappa>0$.  Combining this with \eqref{rza} and the pointwise bound $f_A \leq 1_{F_\kappa} + o(1)$, we conclude that
$$ \mu_G(A) = \int_{\R/\Z} f_A\ d\mu_{\R/\Z} \leq \int_{\R/\Z} 1_{F_\kappa}\ d\mu_{\R/\Z} + o(1) \leq \mu_G(A) + o(1)$$
which implies in particular that 
\begin{equation}\label{mfk}
\mu_{\R/\Z}(F_\kappa) = \mu_G(A) + o(1)
\end{equation}
 and
\begin{equation}\label{rzf}
 \int_{\R/\Z} |1_{F_\kappa} - f_A|\ d\mu_{\R/\Z} = o(1).
\end{equation}
Pulling back to $G$, this implies that
$$ \int_{G} |1_{\phi^{-1}(F_\kappa)} - 1_A|\ d\mu_G = o(1),$$
thus $A$ is asymptotically equivalent to $\phi^{-1}(F_\kappa)$.  Thus to establish the proposition, it suffices to show that $F_\kappa$ is asymptotically equivalent to an arc.

From \eqref{rze}, \eqref{rzf} we have
\begin{equation}\label{rzg}
 \int_{\R/\Z \backslash E} 1_{F_\kappa} * 1_{[0,t]\text{ mod }\Z}\ d\mu_{\R/\Z} = o(1).
\end{equation}
This bound can be used to show that partial sumsets of $F_\kappa$ and $[0,t] \text{ mod } \Z$ are mostly contained in $E$.  However, it does not control the full sumset of these two sets.  To get around this difficulty, we ``smooth'' $F_\kappa$ somewhat by replacing it with a modified set $H_\sigma$.  More precisely,
let $0 < \sigma < t$ be a small fixed quantity, and let $H_\sigma \subset \R/\Z$ be the set $H_\sigma \coloneqq F_\kappa +_{\sigma^2} ([0,\sigma] \text{ mod }\Z)$.  Observe that if $x \in H_\sigma$, then one has the pointwise lower bound $1_{F_\kappa} * 1_{[0,t]\text{ mod }\Z} \geq \sigma^2$ on the arc $x + ([0, t-\sigma] \text{ mod } \Z)$; thus
$$1_{F_\kappa} * 1_{[0,t]\text{ mod }\Z} \geq \sigma^2 1_{H_\sigma + [0, t-\sigma] \text{ mod } \Z}.$$
From this, \eqref{rzg} and Markov's inequality we conclude that all but $o(1)$ in measure of $H_\sigma + ([0, t-\sigma] \text{ mod } \Z)$ lies in $E$.  By \eqref{met}, we conclude that
$$ \mu_{\R/\Z}( H_\sigma + ([0, t-\sigma] \text{ mod } \Z)) \leq \mu_G(A) + t + o(1).$$
On the other hand, from  Corollary \ref{kemp-cor} and \eqref{mfk}, \eqref{muat} one has
$$ \mu_{\R/\Z}(H_\sigma) \geq \min( \mu_{\R/\Z}(F_\kappa) + t - \sigma, 1 ) - 2 \sigma \geq \mu_G(A) - 3 \sigma + o(1).$$
The situation here is reminiscent of that for which the inverse theorem for the Brunn-Minkowski inequality (see \cite{figalli}, \cite{christ}, \cite{christ2}), can be applied, but we are on the circle $\R/\Z$ instead of the line $\R$.  However, as one of the sets involved is an arc, we can use the following elementary argument. As $H_\sigma$ is measurable, it is asymptotically equivalent to some finite union $K$ of arcs.  For each $0 \leq s \leq t-\sigma$, the set $K_s \coloneqq K + ([0,s] \text{ mod } \Z)$ is also a finite union of arcs, with
$$ \mu_G(A) - 3\sigma + o(1) \leq \mu_{\R/\Z}(K_0) \leq \mu_{\R/\Z}(K_{t-\sigma}) \leq \mu_G(A) + t + o(1).$$
It is easy to see that the function $s \mapsto \mu_G(K_s)$ is continuous and piecewise linear, with all slopes being positive integers.  From the fundamental theorem of calculus, we thus see that the slope must in fact equal $1$ for all $s$ in $[0,t-\sigma]$ outside of a set of measure at most $4\sigma+o(1)$.  The slope can only equal one when $K_s$ is an arc, thus $K_s$ must be an arc for some $s \leq 4\sigma+o(1)$.  From the fundamental theorem of calculus again, we have 
$$\mu_{\R/\Z}(K_s) \leq \mu_{\R/\Z}(K_{t-\sigma}) - (t-\sigma-s) \leq \mu_G(A) + 5\sigma+o(1)$$
and thus $K = K_0$ differs by at most $O(\sigma)+o(1)$ in measure from an arc of length $\mu_G(A) + O(\sigma) + o(1)$, where we adopt the convention that implied constants in asymptotic notation are independent of $\sigma$.  This implies that $H_\sigma$ differs by $O(\sigma)+o(1)$ in measure from an arc $I$ of length $\mu_G(A) + O(\sigma) + o(1)$.  Since $1_{F_\kappa} * 1_{[0,\sigma] \text{ mod }\Z}$ is bounded pointwise by $\sigma$, and by $\sigma^2$ outside of $H_\sigma$, we conclude that
$$ \int_{\R/\Z \backslash I} 1_{F_\kappa} * 1_{[0,\sigma] \text{ mod }\Z}\ d\mu_{\R/\Z} \ll \sigma^2 + o(1)$$
which by Fubini's theorem implies that $F_\kappa$ has at most $O(\sigma) + o(1)$ in measure outside of the arc $I - [0,\sigma]$, which has measure $\mu_G(A) + O(\sigma) + o(1)$.  From \eqref{mfk} we conclude that $F_\kappa$ differs from an arc of measure $\mu_G(A)$ by at most $O(\sigma)+o(1)$ in measure.  Sending $\sigma$ to zero sufficiently slowly as $n \to \infty$, we obtain the claim.
\end{proof}

If $(A,B)$ is a critical pair, define an \emph{almost sumset} $A +_{o(1)} B$ of the pair to be any set of the form $A +_\delta B$, where $\delta>0$ is an infinitesimal obeying \eqref{muab-2}.  Clearly at least one almost sumset exists.  The almost sumset is not unique; however, if $\delta >\delta' > 0$ are two infinitesimals obeying \eqref{muab-2}, then we certainly have
$$ A +_\delta B \supset A +_{\delta'} B$$
and hence from Corollary \ref{kemp-cor}
$$ \mu_G(A) + \mu_G(B) + o(1) \geq \mu_G( A +_\delta B ) \geq \mu_G( A +_{\delta'} B ) \geq \mu_G(A) + \mu_G(B) - o(1)$$
and hence $A +_\delta B$ and $A +_{\delta'} B$ are asymptotically equivalent.  Thus, the almost sumset $A +_{o(1)} B$ is well defined up to asymptotic equivalence.  As a first approximation, the reader may think of $A +_{o(1)} B$ as being the full sumset $A+B$; however, we do not use the latter set for technical reasons (it is not stable with respect to asymptotic equivalence).

We now observe the following submodularity property, related to \eqref{submod}:

\begin{lemma}[Submodularity]\label{submod-lemma}  Suppose that $(A,B_1), (A,B_2)$ are critical pairs with 
$$\mu_G(B_1 \cap B_2), 1 - \mu_G(A) - \mu_G(B_1 \cup B_2) \gg 1.$$
Then $(A, B_1 \cap B_2)$ and $(A, B_1 \cup B_2)$ are also critical pairs.
\end{lemma}

The reader may wish to check that the lemma is true in the case when $A,B_1,B_2$ are parallel Bohr sets.  Of course, once Theorem \ref{inv-3} is proven we know that this is essentially the only case in which the hypotheses of the lemma apply, but we cannot use this fact directly as this would be circular.

\begin{proof}  The properties \eqref{muab} for $(A, B_1 \cap B_2)$ and $(A, B_1 \cup B_2)$ are clear from construction, so it suffices to show that
\eqref{muab-2} also holds for these pairs.

By hypothesis, we can find an infinitesimal $\delta>0$ such that
$$ \mu_G( A +_\delta B_1 ) \leq \mu_G(A) + \mu_G(B_1) + o(1)$$
and
$$  \mu_G( A +_\delta B_2 ) \leq \mu_G(A) + \mu_G(B_2) + o(1)$$
(note that we can use the same $\delta$ for both critical pairs $(A,B_1), (A,B_2)$ by increasing one of the $\delta$'s as necessary).  In particular, from the pointwise bound
$$ \min( 1_A * 1_{B_1}, \sqrt{\delta} ) \leq \sqrt{\delta} 1_{A +_\delta B_1} + \delta$$
one has
$$ \int_G \min( 1_A * 1_{B_1}, \sqrt{\delta} )\ d\mu_G \leq \sqrt{\delta}( \mu_G(A) + \mu_G(B_1) + o(1) ) + \delta = \sqrt{\delta}( \mu_G(A) + \mu_G(B_1) + o(1) )$$
and similarly
$$ \int_G \min( 1_A * 1_{B_2}, \sqrt{\delta} )\ d\mu_G \leq \sqrt{\delta}( \mu_G(A) + \mu_G(B_2) + o(1) ).$$
Summing and applying \eqref{a12b} (with the obvious relabeling) together with the inclusion-exclusion identity $\mu_G(B_1)+\mu_G(B_2) = \mu_G(B_1 \cap B_2) + \mu_G(B_1 \cup B_2)$, we conclude that
\begin{align*}
& \int_G \min( 1_A * 1_{B_1 \cap B_2}, \sqrt{\delta} )\ d\mu_G + \int_G \min( 1_A * 1_{B_1 \cup B_2}, \sqrt{\delta} )\ d\mu_G \\
&\quad \leq \sqrt{\delta}( \mu_G(A) + \mu_G(B_1 \cap B_2) + \mu_G(A) + \mu_G(B_1 \cup B_2) + o(1) ).
\end{align*}
On the other hand, from Theorem \ref{ruzsa-thm} we have
$$ \int_G \min( 1_A * 1_{B_1 \cap B_2}, \sqrt{\delta} )\ d\mu_G \geq \sqrt{\delta}( \mu_G(A) + \mu_G(B_1 \cap B_2) - o(1))$$
and similarly
$$ \int_G \min( 1_A * 1_{B_1 \cup B_2}, \sqrt{\delta} )\ d\mu_G \geq \sqrt{\delta}( \mu_G(A) + \mu_G(B_1 \cup B_2) - o(1))$$
Thus we in fact have
$$ \int_G \min( 1_A * 1_{B_1 \cap B_2}, \sqrt{\delta} )\ d\mu_G = \sqrt{\delta}( \mu_G(A) + \mu_G(B_1 \cap B_2) + o(1))$$
and
$$ \int_G \min( 1_A * 1_{B_1 \cup B_2}, \sqrt{\delta} )\ d\mu_G = \sqrt{\delta}( \mu_G(A) + \mu_G(B_1 \cup B_2) + o(1))$$
In particular, we have
$$ \mu_G(A +_{\sqrt{\delta}} (B_1 \cap B_2)) \leq \mu_G(A) + \mu_G(B_1 \cap B_2) + o(1)$$
and
$$ \mu_G(A +_{\sqrt{\delta}} (B_1 \cup B_2)) \leq \mu_G(A) + \mu_G(B_1 \cup B_2) + o(1)$$
We thus obtain \eqref{muab} for $(A,B_1 \cap B_2)$ and $(A,B_1 \cup B_2)$ as desired (with $\delta$ replaced by $\sqrt{\delta}$).
\end{proof}

We can iterate this lemma to obtain

\begin{corollary}\label{iter}  Let $(A,B)$ be a critical pair, and let $\delta>0$ be fixed.  Then there exists a measurable set $C$ with $\mu_G(C) \leq \delta$ such that $(A,C)$ is a critical pair.
\end{corollary}

\begin{proof}  By hypothesis, there exists a fixed $c>0$ such that
$$ \mu_G(A), \mu_G(B), 1 - \mu_G(A) - \mu_G(B) \geq c$$
for $n$ large enough.  For the given $A,B$ and any fixed $\delta>0$, let $P(\delta)$ denote the assertion that there exists $C$ with $\mu_G(C) \leq \min(\mu_G(B),\delta)$ such that $(A,C)$ is a critical pair.  Clearly $P(\delta)$ holds for any $\delta \geq 1-c$, as one can simply take $C = B$.  Now suppose that $P(\delta)$ holds for some $\delta \leq 1-c$, thus there exists $C$ with $\mu_G(C) \leq \delta$ and $(A,C)$ a critical pair.  By Lemma \ref{cont}, one can find $x \in G$ such that $\mu_G( C \cap (x+C) ) = \max( \mu_G(C)^2, \mu_G(C) - c/2 )$.  Observe that
$$ 1 - \mu_G(A) - \mu_G(C \cup (x+C)) \geq 1 - \mu_G(A) - \mu_G(C) - c/2 \geq 1 - \mu_G(A) - \mu_G(B) - c/2 \geq c/2.$$
As $(A,C)$ and $(A,x+C)$ are both critical pairs, we conclude from Lemma \ref{submod-lemma} that $(A, C \cap (x+C))$ is also a critical pair.  Thus $P(\delta')$ holds for all $\delta' \geq \max(\delta^2, \delta-c/2)$.  Iterating this, we conclude that $P(\delta)$ holds for all fixed $\delta>0$, giving the claim.
\end{proof}

As a consequence of this corollary and Proposition \ref{p1}, we may now reduce Theorem \ref{inv-3} to the following variant:

\begin{theorem}[Inverse theorem, reduced form]\label{inv-4}  Let $K$ be a sufficiently large absolute constant.  Suppose that $(A,C)$ is a critical pair such that
\begin{equation}\label{mu}
 \mu_G(A) + K \mu_G(C) < 1
\end{equation}
and
\begin{equation}\label{mu2}
 \mu_G(A) \geq K \mu_G(C).
\end{equation}
Then $C$ is asymptotically equivalent to a Bohr set.
\end{theorem}

One can in fact take $K=10^4$ in our arguments, but the exact value of $K$ will not be of importance to us.  

We now claim that Theorem \ref{inv-3} follows from Theorem \ref{inv-4}.  Indeed, if $(A,B)$ is a critical pair and $K$ is as as in Theorem \ref{inv-4}, then by applying Corollary \ref{iter} with a sufficiently small $\delta$ we may find a critical pair $(A,C)$ obeying \eqref{mu}, \eqref{mu2}.  By Theorem \ref{inv-4}, $C$ is asymptotically equivalent to a Bohr set, which by Proposition \ref{p1} implies that $A$ is asymptotically equivalent to a parallel Bohr set.  But by a second application of Proposition \ref{p1}, we conclude that $B$ is also asymptotically equivalent to a parallel Bohr set, and Theorem \ref{inv-3} follows.

It remains to establish Theorem \ref{inv-4}.  To do this, we first iterate Lemma \ref{submod-lemma} in a different fashion to obtain

\begin{proposition}\label{muto}  Suppose that $(A,B_1), (A,B_2)$ are critical pairs with
\begin{equation}\label{mut}
\mu_G(A) - \mu_G(B_1), 1 - \mu_G(A) - \mu_G(B_1) - \mu_G(B_2) \gg 1.
\end{equation}
Then $(B_1,B_2)$, $(A +_{o(1)} B_1, B_2)$, and $(A, B_1 +_{o(1)} B_2)$ are critical pairs.
\end{proposition}

Recall that $A +_{o(1)} B_1$ and $B_1 +_{o(1)} B_2$ are only defined up to asymptotic equivalence (with the latter only existing because $(B_1,B_2)$ is a critical pair), but this is of no concern here thanks to Lemma \ref{crit-equiv}.  As before, the reader may verify that this claim is easily checked in the case that $A,B_1,B_2$ are parallel Bohr sets.

\begin{proof}  By definition, we can write $A +_{o(1)} B_1$ as $A +_\delta B_1$ for some infinitesimal $\delta>0$ with
\begin{equation}\label{mad}
 \mu_G( A +_\delta B_1 ) = \mu_G(A) + \mu_G(B_1) + o(1).
\end{equation}
Now let $m$ be a fixed large natural number, thus $\delta = \delta_n \leq 1/m$ for $n$ large enough.  From Corollary \ref{kemp-cor} one has
$$ \mu_G(A) + \mu_G(B_1) + o(1) \geq \mu_G( A +_\delta B_1 ) \geq \mu_G( A +_{1/m} B_1 ) \geq \mu_G(A) + \mu_G(B_1) - O(1/\sqrt{m}),$$
where we adopt the convention in this proof that implied constants in the $O()$ and $\ll$ asymptotic notation are independent of $m$.  Hence we have
\begin{equation}\label{aim-o}
 \mu_G( (A +_\delta B_1) \backslash (A +_{1/m} B_1) ) \ll 1/\sqrt{m}
\end{equation}
for $n$ large enough.

Next, we claim there exists a finite set $X_m \subset B_1$ of cardinality at most $m^2$, such that
\begin{equation}\label{aim}
 \mu_G( (A + X_m) \Delta (A +_\delta B_1) ) \ll 1/\sqrt{m}
\end{equation}
for all sufficiently large $n$.
To establish this claim we use the probabilistic method.  Let $x_1,\dots,x_{m^2}$ be chosen independently and uniformly from $B_1$ (using the probability measure $\frac{1}{\mu_G(B_1)} \mu\downharpoonright_{B_1}$ formed by restricting $\frac{1}{\mu_G(B_1)} \mu$ to $B_1$).  Form the random set $X_m \coloneqq \{x_1,\dots,x_{m^2}\}$.  For any $x \in G$, we see that $x \in A+X_m$ precisely when at least one of $x_1,\dots,x_{m^2}$ lie in $x-A$.  By construction, this occurs with probability
$$ 1 - (1-\mu_G( (x-A) \cap B_1 ) / \mu_G(B_1))^{m^2} = 1 - (1-1_A * 1_{B_1}(x) / \mu_G(B_1))^{m^2}.$$
In particular, if $x \in A +_{1/m} B_1$, then $x \in A+X_m$ with probability at least $1 - (1-1/m)^{m^2} = 1 - O( \exp(-m) )$, while if $x \not \in A +_{\delta} B_1$, then $x \in A+X_m$ with probability $o(1)$.  By linearity of expectation (or Fubini's theorem), we conclude that the expected measure of $(A +_{1/m} B_1) \backslash (A+X)$ is $O(\exp(-m))$, while the expected measure of $(A+X_m) \backslash (A +_{\delta} B_1)$ is $o(1)$.  By Markov's inequality, we conclude that there exists a deterministic choice of $X_m$ such that 
$$ \mu_G((A +_{1/m} B_1) \backslash (A+X_m)) \ll \exp(-m)$$
and
$$ \mu_G((A+X_m) \backslash (A +_{\delta} B_1)) \ll o(1)$$
and the claim \eqref{aim} follows from \eqref{aim-o}.

From \eqref{aim}, \eqref{mad} we see in particular that
$$ 
 \mu_G( A + X_m ) \leq \mu_G(A) + \mu_G(B_1) + O(1/\sqrt{m}) $$
for $n$ large enough, and hence by \eqref{mut} we have
$$ 2\mu_G(A) - \mu_G(A + X_m) \gg 1$$
for $n$ large enough.  In particular, we see that for any $x,x' \in X_m$, we have
$$ \mu_G( (A+x) \cap (A+x') ) \gg 1.$$
A similar argument also gives
$$ 1 - \mu_G(A+X_m) - \mu_G(B_2) \gg 1.$$

By translation invariance, $(A+x, B_2)$ is a critical pair for each $x \in X_m$.  Applying Lemma \ref{submod-lemma} at most $m^2$ times and using the above estimates to verify the hypotheses of that lemma, we conclude that
$(A+X_m, B_2)$ is also a critical pair.  

The set $A+X_m$ is not quite asymptotically equivalent to $A +_\delta B_1$; but by \eqref{aim} and a diagonalisation argument we see that $A + X_{m_n}$ is asymptotically equivalent to $A +_\delta B_1$ if $m_n$ goes to infinity sufficiently slowly as $n \to \infty$.   As each $(A + X_m, B_2)$ is a critical pair, $(A + X_{m_n}, B_2)$ will also be a critical pair for $m_n$ going to infinity sufficiently slowly.  Applying Lemma \ref{crit-equiv}, we conclude that $(A +_\delta B_1, B_2)$ is a critical pair, giving the second of the three claims of the proposition.

Write $C \coloneqq A +_\delta B_1$, thus (as $(A,B_1)$ is a critical pair)
\begin{equation}\label{cb1}
\mu_G(C) = \mu_G(A) + \mu_G(B_1) + o(1).
\end{equation}
As $(C,B_2)$ is a critical pair, there exists an infinitesimal $\delta'>0$ such that
\begin{equation}\label{cb2}
 \mu_G( C +_{\delta'} B_2 ) \leq \mu_G(C) + \mu_G(B_2) + o(1).
\end{equation}
Set
$$ \sigma \coloneqq (\delta + \delta')^{1/3},$$
thus $\sigma>0$ is infinitesimal, and write $D \coloneqq B_1 +_\sigma B_2$.  We now consider the expression
\begin{equation}\label{1bd}
 \int_{G \backslash (C +_{\delta'} B_2)} 1_A * 1_D\ d\mu_G.
\end{equation}
By definition of $D$, we have the pointwise estimate
$$ 1_D \leq \frac{1}{\sigma} 1_{B_1} * 1_{B_2}$$
and hence we can bound \eqref{1bd} by
$$
\frac{1}{\sigma} \int_{G \backslash (C +_{\delta'} B_2)} 1_A * 1_{B_1} * 1_{B_2}\ d\mu_G$$
(here we implicitly use the fact that convolution is associative).  On the other hand, by definition of $C$ we have the pointwise estimate
$$ 1_A * 1_{B_1} \leq \delta + 1_C $$
and hence we can bound \eqref{1bd} by
$$ \frac{\delta}{\sigma} + \frac{1}{\sigma} \int_{G \backslash (C +_{\delta'} B_2)} 1_C * 1_{B_2}\ d\mu_G.$$
Since $1_C * 1_{B_2}$ is bounded by $\delta'$ outside of $C +_{\delta'} B_2$, we conclude that
$$
 \int_{G \backslash (C +_{\delta'} B_2)} 1_A * 1_D\ d\mu_G \leq \frac{\delta + \delta'}{\sigma} = \sigma^2.$$
By Markov's inequality, we conclude that
$$ \mu_G( (A +_\sigma D) \backslash (C +_{\delta'} B_2) ) \leq \sigma = o(1)$$
and hence by \eqref{cb1}, \eqref{cb2} one has
$$ \mu_G( A +_\sigma D ) \leq \mu_G(A) + \mu_G(B_1) + \mu_G(B_2) + o(1).$$
On the other hand, from two applications of Corollary \ref{kemp-cor} (and \eqref{mut}) one has
$$ \mu_G(D) \geq \mu_G(B_1) + \mu_G(B_2) - o(1)$$
and
$$ \mu_G(A +_\sigma D ) \geq \min( \mu_G(A) + \mu_G(D), 1 ) - o(1).$$
By \eqref{mut}, these bounds can only be consistent if
$$ \mu_G(D) = \mu_G(B_1) + \mu_G(B_2) + o(1)$$
and
$$ \mu_G(A +_\sigma D) = \mu_G(A) + \mu_G(D) + o(1)$$
so that $(B_1,B_2)$ and $(A,D)$ are both critical pairs, giving the final two claims of the proposition.
\end{proof}

\begin{remark} It is important in the above argument that we work with the almost sumset $A +_{o(1)} B_1$ rather than $A + B_1$, as we do not know how to approximate the latter set by sumsets $A+X$ of $A$ with a finite set $X$.  As a consequence, even if one is only interested in Theorem \ref{inv-1}, the proof methods of this paper only seem to work if one first proves the stronger claim in Theorem \ref{inv-2}.
\end{remark}

Now we can finish the proof of Theorem \ref{inv-4}.  Let $K$ and $(A,C)$ be as in the statement of that theorem.
From \eqref{mu}, \eqref{mu2}, Proposition \ref{muto}, we see that $(C,C)$ is a critical pair, and there exists a set $C_2 = C +_{o(1)} C$ of measure $\mu_G(C_2) = 2\mu_G(C)+o(1)$ such that $(A,C_2)$ is a critical pair.  By further iteration of Proposition \ref{muto} using \eqref{mu}, \eqref{mu2}, we in fact can find a set $C_k$ of measure 
\begin{equation}\label{muck}
\mu_G(C_k) = k \mu_G(C) + o(1)
\end{equation}
for each even number $k=2,4,\dots,K-2$ such that $(A,C_k)$ is a critical pair, and for each even $k=2,\dots,K-4$, $(C_2,C_k)$ is a critical pair with
\begin{equation}\label{ck2}
 C_{k+2} = C_2 +_{o(1)} C_{k}.
\end{equation}

We now use the linear growth \eqref{muck} to approximate $C$ by a Bohr set, using an argument of Schoen \cite{schoen} (later employed by Green and Ruzsa \cite{rect}, \cite{green}) to locate the relevant character $\phi$.  The character $\chi$ that this argument produces may not necessarily be the one used to construct the Bohr set, but it turns out that it is closely related to that character (one may have to divide the initial character by a bounded natural number).

From \eqref{ck2} we see that $1_{C_2} * 1_{C_k}$ is bounded pointwise by $1_{C_{k+2}} + o(1)$ for every even $k=2,\dots,K-4$.
By induction we then see that for every $k=1,\dots,\frac{K}{2}-2$, the $k$-fold convolution
$$ 1_{C_2}^{*k} = 1_{C_2} * \dots * 1_{C_2} $$
is bounded pointwise by $1_{C_{2k}} + o(1)$.  In particular, by Fubini's theorem we have
$$ \int_{C_{2k}} 1_{C_2}^{*k}\ d\mu_G \geq \mu_G(C_2)^k - o(1);$$
from \eqref{muck} and Cauchy-Schwarz, we conclude that
$$ \int_G (1_{C_2}^{*k})^2\ d\mu_G \geq \frac{1}{2k} \mu_G(C_2)^{2k-1} - o(1).$$
On the other hand, by Plancherel's theorem we may write
$$ \int_G (1_{C_2}^{*k})^2\ d\mu_G = \sum_{\phi \in \hat G} |\hat 1_{C_2}(\phi)|^{2k}$$
where (as in Section \ref{cor-sec}) the Pontryagin dual $\hat G$ is the collection of all continuous homomorphisms (characters) $\phi: G \to \R/\Z$, and $\hat 1_{C_2}(\phi)$ are the Fourier coefficients
$$ \hat 1_{C_2}(\phi) \coloneqq \int_G 1_{C_2}(x) e^{-2\pi i \phi(x)}\ d\mu_G(x).$$
The contribution of the trivial homomorphism $0$ to the above sum is $\mu_G(C_2)^{2k}$, which will be smaller than half the main term if $k \leq K/8$, thanks to \eqref{mu}.  We conclude that
$$ \sum_{\phi \in \hat G: \phi \neq 0} |\hat 1_{C_2}(\phi)|^{2k} \geq \frac{1}{4k} \mu_G(C_2)^{2k-1} - o(1)$$
for $k \leq K/8$ and $n$ large enough.  On the other hand, from Plancherel's theorem we have
$$ \sum_{\phi \in \hat G} |\hat 1_{C_2}(\phi)|^{2} = \mu_G(C_2).$$
We conclude that there exists a non-zero continuous homomorphism $\phi: G \to \R/\Z$ such that
$$ |\hat 1_{C_2}(\phi)| \geq \frac{1}{(4k)^{\frac{1}{2k-2}}} \mu_G(C_2) - o(1).$$
Applying this with $k=\left \lfloor \frac{K}{8} \right \rfloor$, we conclude in particular that
\begin{equation}\label{hac}
 |\hat 1_{C_2}(\phi)| \geq \left(1 - O\left(\frac{\log K}{K}\right)\right) \mu_G(C_2) - o(1),
\end{equation}
where we adopt the convention that implied constants in the $O()$ notation are independent of $K$.
The image $\phi(G)$ of $G$ is a non-trivial connected subgroup of $\R/\Z$, and thus must be all of $\R/\Z$; thus $\phi$ is surjective.

\begin{remark}\label{remo}  A good example to keep in mind here is if $G = \R/\Z$, $\phi:\R/\Z \to \R/\Z$ is a character $\phi(x) \coloneqq mx$ for some natural number $m \ll 1$, $C = [0,c] \text{ mod } \Z$, and $C_k = [0, kc] \text{ mod } \Z$ for $k=1,\dots,K$ and some small $c\gg 1$ (in particular $c < \frac{1}{10Km}$, say).  In this case we of course have $\mu_G(C) = c$.  Note that while $C$ is a Bohr set, the relevant character here is not $\phi$, but rather the quotient $\frac{1}{m} \phi: x \mapsto x$ of $\phi$ by $m$.  As such, we will need to perform such a quotienting step later in the argument.
\end{remark}

Since $C_2 = C +_{o(1)} C$, we have
$$ \int_{G \backslash C_2} 1_C * 1_C\ d\mu_G = o(1).$$
By Fubini's theorem, the left-hand side may be rewritten as
$$ \int_C \mu_G( (x+C) \backslash C_2)\ d\mu_G(x)$$
and hence by Markov's inequality, there exists a subset $C'$ of $C$ asymptotically equivalent to $C$ such that
\begin{equation}\label{mucc}
 \mu_G( (x+C) \backslash C_2) = o(1) 
\end{equation}
for all $x \in C'$.

From \eqref{hac}, there exists $\theta \in \R/\Z$ such that 
$$ \mathrm{Re} e^{2\pi i\theta} \hat 1_{C_2}(\phi) \geq \left(1 - O\left(\frac{\log K}{K}\right)\right) \mu_G(C_2) - o(1)$$
which we rearrange as
$$ \int_{C_2} \left(1 - \cos(2\pi (\theta - \phi(y)))\right)\ d\mu_G(y) \ll \frac{\log K}{K} \mu_G(C_2) + o(1).$$
From \eqref{mucc} and \eqref{muck}, we conclude in particular that for every $x \in C'$, one has
$$ \int_{x+C} \left(1 - \cos(2\pi (\theta - \phi(y)))\right)\ d\mu_G(y) \ll \frac{\log K}{K} \mu_G(C) + o(1)$$
and hence by change of variables
$$ \int_{C} \left(1 - \cos(2\pi (\theta - \phi(x) - \phi(y)))\right)\ d\mu_G(y) \ll \frac{\log K}{K} \mu_G(C) + o(1),$$
which by Cauchy-Schwarz implies that
$$ \int_{C} \left(1 - \cos(2\pi (\theta - \phi(x) - \phi(y)))\right)^{1/2}\ d\mu_G(y) \leq \left(\frac{\log K}{K}\right)^{1/2} \mu_G(C) + o(1);$$
noting the trigonometric identity
$$ |1 - e^{i \alpha}| = \sqrt{2(1-\cos(\alpha))}$$
we conclude that
$$ \int_{C} \left|1 - e^{2\pi i(\theta - \phi(x) - \phi(y))}\right|\ d\mu_G(y) \ll \left(\frac{\log K}{K}\right)^{1/2}  \mu_G(C) + o(1).$$
From the triangle inequality, we conclude that for any $x,x' \in C$, one has
$$ \int_{C} \left|e^{2\pi i(\theta - \phi(x') - \phi(y))} - e^{2\pi i(\theta - \phi(x) - \phi(y))}\right|\ d\mu_G(y) \ll \left(\frac{\log K}{K}\right)^{1/2} \mu_G(C) + o(1).$$
But the left-hand side simplifies to $2\mu_G(C) |\sin(\pi(\phi(x) - \phi(x')))|$, thus
$$ |\sin(\pi(\phi(x) - \phi(x')))| \ll \left(\frac{\log K}{K}\right)^{1/2} + o(1)$$
for all $x,x' \in C'$.  Thus, if $\| \alpha \|_{\R/\Z}$ denotes the distance of $\alpha$ to the nearest integer, with the associated metric $d_{\R/\Z}(\alpha,\beta) \coloneqq \| \alpha - \beta \|_{\R/\Z}$ on $\R/\Z$, then $\phi(C')$ has diameter $O( (\log K/K)^{1/2} )$ with respect to this metric.  For $K$ large enough (in fact one can check that $K = 10^4$ would suffice), we conclude that there exists $\alpha_0 \in \R/\Z$ such that
$$ \| \phi(x) - \alpha_0 \|_{\R/\Z} < \frac{1}{10}$$
for all $x \in C'$.

Note that we have the freedom to translate $C$ (and $C'$) by an arbitrary shift $x$ in $G$ (shifting $C_{2k}$ by $2kx$ accordingly) without affecting any of the above properties. From this and the surjectivity of $\phi$, we may assume without loss of generality that $\alpha_0=0$, thus
\begin{equation}\label{st}
 \| \phi(x) \|_{\R/\Z} < \frac{1}{10}
\end{equation}
for all $x \in C'$.

Recall the pushforward map $\phi_*: L^2( G ) \mapsto L^2(\R/\Z)$ from the proof of Proposition \ref{p1}.  If we write
$$ f_{C'} \coloneqq \phi_*( 1_{C'} )$$
and 
$$ f_{C_2} \coloneqq \phi_*( 1_{C_2} )$$
then $f_{C'}, f_{C_2}$ are (up to almost everywhere equivalence) functions on $\R/\Z$ taking values in $[0,1]$, with
\begin{equation}\label{rz}
 \int_{\R/\Z} f_{C'}\ d\mu_{\R/\Z} = \mu_G(C) + o(1) 
\end{equation}
and similarly
\begin{equation}\label{rz2}
 \int_{\R/\Z} f_{C_2}\ d\mu_{\R/\Z} = \mu_G(C_2) = 2 \mu_G(C) + o(1).
\end{equation}
From \eqref{st} one has that $f_{C'}$ is supported in the arc $[-\frac{1}{10},\frac{1}{10}] \text{ mod } \Z$.  Let $\tau \coloneqq \|f_{C'}\|_\infty$ denote the essential supremum of $f_{C'}$; since $\mu_G(C) \gg 1$, we have $1 \ll \tau \leq 1$.

\begin{remark}  Continuing the example in Remark \ref{remo}, taking $C' = C$, we would have $f_{C'} = \frac{1}{m} 1_{[0,mc] \text{ mod } \Z}$ and 
$f_{2C} = \frac{1}{m} 1_{[0,2mc] \text{ mod } \Z}$.
\end{remark}

If $x \in C'$, then by \eqref{mucc} $1_{x+C}$, and hence $1_{x+C'}$, is bounded by $1_{C_2}$ plus a function of $L^1(G,\mu_G)$ norm $o(1)$.  Applying $\phi_*$, we conclude that the translate $f_{C'}(\cdot-\phi(x))$ is bounded by $f_{2C} \coloneqq \phi_*(1_{2C})$ plus a function of $L^1(\R/\Z,\mu_{\R/\Z})$ norm $o(1)$. Applying Markov's inequality, we conclude that
for any $t \gg 1$, the set $\phi(x) + \{ f_{C'} \geq t \}$ is contained in the union of $\{ f_{2C} \geq t-o(1) \}$ and a set of measure $o(1)$.  Thus
$$ \int_{f_{2C} \leq t - o(1)} 1_{\{ f_{C'} \geq t \}}( \alpha - \phi(x) )\ d\mu_{\R/\Z}(\alpha) = o(1)$$
for all $x \in C'$.  Integrating over $x$, we conclude that
$$ \int_{\R/\Z} \int_{f_{2C} \leq t - o(1)} 1_{\{ f_{C'} \geq t \}}( \alpha - \beta )\ d\mu_{\R/\Z}(\alpha) f_{C'}(\beta)\ d\mu_{\R/\Z}(\beta) = o(1)$$
or equivalently that
$$ \int_{f_{2C} \leq t - o(1)} 1_{\{ f_{C'} \geq t \}} * f_{C'}\ d\mu_{\R/\Z} = o(1).$$
In particular, for any fixed $s>0$, one has
\begin{equation}\label{op}
 \int_{f_{2C} \leq t - o(1)} 1_{\{ f_{C'} \geq t \}} * 1_{\{ f_{C'} \geq s \}}\ d\mu_{\R/\Z} = o(1).
\end{equation}
Comparing this with Corollary \ref{kemp-cor}, and recalling that $\{ f_{C'} \geq s \}$ and $\{ f_{C'} \geq t \}$ are both contained in $[-\frac{1}{10},\frac{1}{10}] \text{ mod } \Z$ and thus have measure at most $1/5$, we conclude that
$$ \mu_{\R/\Z}( \{ f_{2C} \geq t-o(1) \} ) \geq \mu_{\R/\Z}( \{ f_{C'} \geq t \} ) + \mu_{\R/\Z}( \{f_{C'} \geq s \} ) - o(1)$$
whenever $t,s < \tau$ (so that the sets on the right-hand side are non-empty\footnote{A previous version of this paper neglected to address this rather important issue that Corollary \ref{kemp-cor} breaks down when one of the sets involved is empty. We regret this oversight.ix }).
Integrating over $t$, we conclude that
$$ \int_{\R/\Z} f_{2C}\ d\mu_{\R/\Z} \geq \int_{\R/\Z} f_{C'}\ d\mu_{\R/\Z} + \tau \mu_{\R/\Z}( \{f_{C'} \geq s \} ) - o(1)$$
for any $1 \ll s < \tau$,
and hence by \eqref{rz}, \eqref{rz2} we conclude that
$$ \tau \mu_{\R/\Z}( \{f_{C'} \geq s \} ) \leq \mu_G(C) + o(1)$$
for every fixed $0 < s \ll \tau$.  Diagonalising, we conclude that there exists an infinitesimal $\eps>0$ such that
$$ \tau \mu_{\R/\Z}( \{f_{C'} \geq \eps \} ) \leq \mu_G(C) + o(1).$$
Write 
\begin{equation}\label{sfc-0}
S \coloneqq \{ f_{C'} \geq \eps \}.
\end{equation}
Then from \eqref{rz} we have
$$ \mu_G(C) + o(1) = \int_{\R/\Z} f_{C'}\ d\mu_{\R/\Z} \leq \int_S f_{C'}\ d\mu_{\R/\Z} + o(1) \leq \tau \mu_{\R/\Z}(S) + o(1) \leq \mu_G(C) + o(1)$$
and thus
$$ \int_S f_{C'}\ d\mu_{\R/\Z} = \tau \mu_{\R/\Z}(S) + o(1) = \mu_G(C) + o(1)$$ 
so in particular
\begin{equation}\label{Smes}
\mu_{\R/\Z}(S)  = \tau^{-1} \mu_G(C) + o(1)
\end{equation}
and
\begin{equation}\label{sfc}
 \int_S (\tau-f_{C'})\ d\mu_{\R/\Z} = o(1).
\end{equation}
Hence by Markov's inequality, one has
$$ \mu_{\R/\Z}( \{ f_{C'} \geq t \} ) = \mu_{\R/\Z}(S) + o(1) $$
whenever $t, \tau-t \gg 1$. Using \eqref{op}, we conclude that
$$ \int_{f_{2C} \leq t} 1_S * 1_S\ d\mu_{\R/\Z} = o(1)$$
whenever $t, \tau-t \gg 1$, and hence by diagonalising there exists an infinitesimal $\eps' > 0$ such that
$$ \int_{f_{2C} \leq \tau-\eps'} 1_S * 1_S\ d\mu_{\R/\Z} = o(1).$$
Let $S_2 \coloneqq \{ f_{2C} > \tau-\eps'\} \cap ([-\frac{1}{5},\frac{1}{5}] \text{ mod } \Z)$, then since $1_S * 1_S$ is supported in $[-\frac{1}{5},\frac{1}{5}] \text{ mod } \Z$, one has
\begin{equation}\label{rss}
 \int_{\R/\Z \backslash S_2} 1_S * 1_S\ d\mu_{\R/\Z}= o(1),
\end{equation}
while from \eqref{rz2}, \eqref{Smes}, and Markov's inequality one has
$$ \mu_{\R/\Z}(S_2) \leq \tau^{-1} \int_{\R/\Z} f_{C_2}\ d\mu_{\R/\Z} + o(1) \leq 2 \mu_{\R/\Z}(S) + o(1).$$
Our strategy is to first work on the structure of $f_{C'}$ and $f_{C_2}$ (in particular, to show that these functions are basically indicator functions of arcs multiplied by $\tau$), and then return to the structural classification of $C'$ once this is done.

\begin{remark}\label{remo-2}  Again continuing Remark \ref{remo}, we would essentially have $\tau = m^{-1}$, $S = [0,mc] \text{ mod } 1$, and 
$S_2 = [0,2mc] \text{ mod } 1$.
\end{remark}

If we let $\tilde S \subset [-\frac{1}{10},\frac{1}{10}]$ and $\tilde S_2 \subset [-\frac{1}{5},\frac{1}{5}]$ be the lifts of $S, S_2$ respectively from $\R/\Z$ to $\R$, then we have
\begin{equation}\label{roll}
 \int_{\R \backslash \tilde S_2} 1_{\tilde S} * 1_{\tilde S}\ dm = o(1)
\end{equation}
and
\begin{equation}\label{mt2}
 m(\tilde S_2) \leq 2 m(\tilde S) + o(1),
\end{equation}
where $m$ denotes Lebesgue measure on $\R$.
Also $m(\tilde S) = \mu_{\R/\Z}(S) = \tau^{-1} \mu_G(C) + o(1) \gg 1$.  This type of situation (a near-saturation of the Riesz-Sobolev inequality) was studied by Christ \cite{christ}, \cite{christ2}.  We were not able to directly apply the results from those papers, as this is an endpoint case (the parameter $\eta$ in those papers would be set to $o(1)$ here).  However, we can use the following variant of the arguments in those papers.  The left-hand side of \eqref{roll} can be rearranged as
$$ \int_{\tilde S} m( (x + \tilde S) \backslash \tilde S_2 )\ dm(x)$$
so by Markov's inequality, one can find a subset $\tilde S'$ of $\tilde S$ with $m(\tilde S') = m(\tilde S) - o(1)$ such that
\begin{equation}\label{aod}
 m( (x + \tilde S) \backslash \tilde S_2 ) = o(1)
\end{equation}
for all $x \in \tilde S'$.  By inner regularity we may also take $\tilde S'$ to be compact.

Let $0 < \sigma \leq 1/4$ be a small fixed parameter.  As the primitive $x \mapsto \int_{-\infty}^x 1_{\tilde S'}\ dm$ is continuous and non-decreasing, and constant outside of $\tilde S'$, we can find real numbers $-\frac{1}{10} \leq a < b \leq \frac{1}{10}$ in $\tilde S'$ such that
$$ m( (-\infty,a) \cap \tilde S' ) = m( (b, +\infty) \cap \tilde S' ) = \sigma m(\tilde S).$$
Thus, if one defines $\tilde S_* := [a,b] \cap \tilde S$, then
$$ m(\tilde S_*) = (1-2\sigma) m(\tilde S) + o(1)$$
which in particular forces $b-a \gg 1$, where we adopt the convention that implied constants in the asymptotic notation are independent of $\sigma$.
From \eqref{aod} we have
$$ m( (a + \tilde S_*) \backslash \tilde S_2 ), m( (b + \tilde S_*) \backslash \tilde S_2 ) = o(1)$$
and hence all but $o(1)$ in measure of the set $\{a,b\} + \tilde S_*$ is contained in $\tilde S_2$.  But the sets $a + \tilde S_*$ and $b + \tilde S_*$ are essentially disjoint, thus $\{a,b\} + \tilde S_*$ has measure $(2-4\sigma) m(\tilde S) + o(1)$.  From \eqref{mt2} we conclude that all but $4 \sigma m(\tilde S) + o(1)$ in measure of $\tilde S_2$ is contained in $\{a,b\} + \tilde S_*$.  Since $1_{\tilde S} * 1_{\tilde S}$ is bounded pointwise by $m(\tilde S)$, we conclude from \eqref{roll} that
$$ \int_{\R \backslash (\{a,b\} + \tilde S_*)} 1_{\tilde S} * 1_{\tilde S} \leq 4 \sigma m(\tilde S)^2 + o(1)$$
and in particular
$$ \int_{\R \backslash (\{a,b\} + \tilde S_*)} 1_{\tilde S_*} * 1_{\tilde S_*} \ll \sigma m(\tilde S_*)^2 + o(1).$$
We now project $\R$ to the circle $T \coloneqq \R / (b-a) \Z$, and let $S_*$ be the projection of $\tilde S_* \subset [a,b]$ to that circle.  Then $\mu_{T}(S_*) = \frac{1}{b-a} m(\tilde S_*)$, while $1_{\tilde S_*} * 1_{\tilde S_*}$ is supported on $[2a,2b] = \{a,b\} + [a,b]$.  As $[a,b]$ is essentially a fundamental domain for $T$, we conclude that
$$ \int_{T \backslash S_*} 1_{S_*} * 1_{S_*}\ d\mu_{T} \ll \sigma \mu_{T}(S_*)^2 + o(1)$$
(note that the normalising factors of $\frac{1}{b-a}$ on both sides cancel each other out).
Since
$$ \int_{S_*} \min( 1_{S_*} * 1_{S_*}, \sqrt{\sigma} \mu_{T}(S_*) ) \ d\mu_{T} \leq \sqrt{\sigma} \mu_{T}(S_*)^2,$$
we conclude that
$$ \int_{T} \min( 1_{S_*} * 1_{S_*}, \sqrt{\sigma} \mu_{T}(S_*) ) \ d\mu_{T} \leq (\sqrt{\sigma} + O(\sigma) ) \mu_{T}(S_*)^2 + o(1).$$
On the other hand, from Theorem \ref{ruzsa-thm} one has
$$ \int_{T} \min( 1_{S_*} * 1_{S_*}, \sqrt{\sigma} \mu_{T}(S_*) ) \ d\mu_{T} \geq 
\sqrt{\sigma} \mu_{T}(S_*) \min( 2 \mu_{T}(S_*) - \sqrt{\sigma}, 1 ) $$
and hence
$$ \min( 2 \mu_{T}(S_*) - \sqrt{\sigma}, 1 )  \leq \mu_{T}(S_*) + O(\sqrt{\sigma} ) + o(1).$$
Since $\mu_{T}(S_*) \gg 1$, we conclude on taking $\sigma$ small enough that
$$ \mu_{T}(S_*) \geq 1 - O(\sqrt{\sigma}) - o(1)$$
and thus $\tilde S_*$ occupies all but $O(\sqrt{\sigma}) + o(1)$ of the interval $[a,b]$ in measure.  Since $\tilde S_*$ occupies all but $O(\sigma) + o(1)$ in measure of $\tilde S$, we conclude that
$$ m( \tilde S \Delta [a,b] ) \ll \sqrt{\sigma} + o(1)$$
and hence 
$$ \mu_{\R/\Z}( S \Delta ([a,b] \text{ mod } \Z) ) \ll \sqrt{\sigma} + o(1).$$
By sending $\sigma$ sufficiently slowly to zero, rather than being fixed, we have thus located a compact arc $I = [a,b] \text{ mod } \Z$ with $-\frac{1}{10} \leq a < b \leq \frac{1}{10}$ such that
$$ \mu_{\R/\Z}(S \Delta I ) = o(1).$$
From \eqref{sfc-0}, \eqref{sfc} one has
$$ \int_{\R/\Z} |f_{C'} - \tau 1_S|\ d\mu_{\R/\Z} = o(1)$$
and hence by the triangle inequality
\begin{equation}\label{fci}
 \int_{\R/\Z} |f_{C'} - \tau 1_{I}|\ d\mu_{\R/\Z} = o(1).
\end{equation}
From \eqref{rz} we now have
\begin{equation}\label{mugc-new}
 \mu_G(C) = \tau (b-a) + o(1)
\end{equation}
so in particular $b-a \gg 1$.
Also, from \eqref{rss} we now see that
$$ \mu_{\R/\Z}(S_2 \Delta I ) = o(1),$$
where $2I \coloneqq [2a, 2b] \text{ mod } \Z$, and hence by \eqref{rz2} and the definition of $S_2$ we have
\begin{equation}\label{fci-2}
 \int_{\R/\Z} |f_{C_2} - \tau 1_{2I}|\ d\mu_{\R/\Z} = o(1).
\end{equation}
To summarise so far, we have obtained a satisfactory description of the functions $f_{C'}, f_{C_2}$, namely that they are equal to $\tau 1_I$ and $\tau 1_{2I}$ respectively up to negligible errors.  If $\tau=1$ we would now be quickly done, as we could then show that $C'$ is asymptotically equivalent to the Bohr set $\phi^{-1}(I)$, which would then imply the same statement for $C$, as required for Theorem \ref{inv-4}.  Unfortunately, as Remark \ref{remo-2} shows, $\tau$ can be less than $1$, and we will need to ``quotient'' the character $\phi$ by a natural number $m$ (which will turn out to be very close to $\tau^{-1}$) to deal with this issue.  

We turn to the details.  Let $C'' \coloneqq C' \cap \phi^{-1}(I)$.  From \eqref{fci} we have
$$ \mu_G(C' \backslash C'') = \int_{\R/\Z \backslash I} f_{C'}\ d\mu_{\R/\Z} = o(1)$$
and so $C''$ is asymptotically equivalent to $C'$ and hence to $C$.  
As $C''$ is contained in $\phi^{-1}(I)$, the difference set $C''-C''$ is contained in $\phi^{-1}(I-I) = \{ x \in G: \| \phi(x) \|_{\R/\Z} \leq b-a\}$.  Crucially, we have the following lower bound:

\begin{lemma}\label{lodo}  For every $x \in C'' - C''$, we have
$$ 1_{C''} * 1_{-C''}(x) \geq \tau (b-a-\|\phi(x)\|_{\R/\Z}) - o(1).$$ 
\end{lemma}

\begin{proof}  As discussed above, $x$ lies in $\phi^{-1}(I-I)$, so $\phi(x)$ lies in the interval $[a-b, b-a] \text{ mod } \Z$.  As $1_{C''} * 1_{-C''}$ is an even function, we may assume without loss of generality that $\phi(x) = h \text{ mod } \Z$ for some $0 \leq h \leq b-a$.  By construction, we have $x = y-z$ for some $y,z \in C''$, then $\phi(z) = s \hbox{ mod } \Z$ and $\phi(y) = s+h \hbox{ mod } z$ for some $a \leq s \leq b-h$.  Since $y,z \in C'' \subset C'$, we see from \eqref{mucc} that
$$ \mu_G( (y+C'') \backslash C_2 ), \mu_G( (z+C'') \backslash C_2 ) = o(1).$$
In particular, the sets
$$ y+(C'' \cap \phi^{-1}([a,b-h]\text{ mod } \Z)), z + (C'' \cap \phi^{-1}([a+h,b] \text{ mod } \Z))$$
are both contained in the set $C_2 \cap \phi^{-1}([a+h+s,b+s]\text{ mod } \Z)$, outside of a set of measure $o(1)$.  But by \eqref{fci}, the set $y+(C'' \cap \phi^{-1}([a,b-h]\text{ mod } \Z))$ has measure
$$ \int_{[a,b-h]\text{ mod } \Z} f_{C'}\ d\mu_{\R/\Z} = \tau(b-a-h) + o(1)$$
and similarly $z + (C'' \cap \phi^{-1}([a+h,b] \text{ mod } \Z))$ also has measure $b-a-h+o(1)$.  By \eqref{fci-2}, the set $C_2 \cap \phi^{-1}([a+h+s,b+s]\text{ mod } \Z)$ has measure
$$ \int_{[a+h+s, b+s]} f_{C_2}\ d\mu_{\R/\Z} = \tau(b-a-h) + o(1).$$
By the inclusion-exclusion principle, we conclude that
$$ \mu_G( y+(C'' \cap \phi^{-1}([a,b-h]\text{ mod } \Z)), z + (C'' \cap \phi^{-1}([a+h,b] \text{ mod } \Z)) ) \geq \tau(b-a-h) - o(1).$$
Since the left-hand side is at least $1_{C''} * 1_{-C''}(y-z) = 1_{C''} * 1_{-C''}(x)$, the claim follows.
\end{proof}

As a consequence, we can now obtain a local additive closure property for $C''-C''$:

\begin{corollary}\label{raz}  There is a positive quantity $\kappa = o(1)$ with the property that whenever $x,y \in C'' - C''$ with
$$ \| \phi(x) \|_{\R/\Z} + \| \phi(y) \|_{\R/\Z} \leq b-a-\kappa$$
then $x+y \in C''-C''$.
\end{corollary}

\begin{proof}  Let $\kappa = o(1)$ be an infinitesimal to be chosen later.  From the preceding lemma we have
$$ \mu_G(C'' \cap (x+C'')) = 1_{C''} * 1_{-C''}(x) \geq \tau (b-a-\|\phi(x)\|_{\R/\Z}) - o(1)$$ 
and
$$ \mu_G((x+C'') \cap (x+y+C'')) = 1_{C''} * 1_{-C''}(y) \geq \tau (b-a-\|\phi(y)\|_{\R/\Z}) - o(1)$$ 
while from \eqref{mugc-new} we have
$$ \mu_G(C'') = \tau(b-a) + o(1).$$
From the triangle inequality, we conclude that
$$ \mu_G(C'' \cap (x+y+C'')) = \tau (b-a-\|\phi(x)\|_{\R/\Z} - \| \phi(y) \|_{\R/\Z}) - o(1).$$ 
For $\kappa$ going to zero sufficiently slowly, the right-hand side is positive, and hence $x+y \in C'' - C''$ as desired.
\end{proof}

The kernel $\phi^{-1}(0)$ of $\phi$ is a compact subgroup of $G$.  Set $H \coloneqq (C''-C'') \cap \phi^{-1}(0)$, then $H$ is compact and symmetric around the origin.  By the above corollary, it is also closed under addition; thus $H$ is a compact subgroup of $\phi^{-1}(0)$.  By a further application of the above corollary, we see that whenever $x \in C''-C''$ is such that $\|\phi(x)\|_{\R/\Z} \leq b-a-\kappa$, then
$$ (C''-C'') \cap (x+\phi^{-1}(0)) = x + H.$$
By yet another application of this corollary, we see that the set $E \coloneqq \phi(C''-C'') \cap ([a-b+\kappa,b-a-\kappa] \text{ mod } \Z)$ is locally closed under addition in the sense that
$$ (E + E) \cap ([a-b+\kappa,b-a-\kappa] \text{ mod } \Z) \subset E.$$
By \eqref{fci} we see that $E$ occupies all but $o(1)$ of the arc $[a-b+\kappa,b-a-\kappa] \text{ mod } \Z$; from the above inclusion and the pigeonhole principle we conclude that $E$ contains the interval $J \coloneqq [a-b+\kappa',b-a-\kappa'] \text{ mod } \Z$ for some infinitesimal $\kappa' > \kappa$.  We thus see that we have a representation of the form
$$ (C''-C'') \cap \phi^{-1}(J) = \bigcup_{s \in J} \psi(s)$$
where for each $s \in J$, $\psi(s) \in G/H$ is a coset of $H$ that lies in the coset $\phi^{-1}(s)$ of $\phi^{-1}(0)$.

Since $(C''-C'') \cap \phi^{-1}(J)$ is a compact set of positive measure in $G$, $\psi(J)$ is a compact set of positive measure in the quotient group $G/H$, which is a compact connected group.  The character $\phi: G \to \R/\Z$ descends to a character $\tilde \phi: G/H \to \R/\Z$.  The translates $\psi(J) + h$ for $h \in \tilde \phi^{-1}(0)$ are all disjoint, and hence the kernel $\tilde \phi^{-1}(0)$ must be finite since $G/H$ has finite measure.  If $m$ is the cardinality of this kernel, then $\tilde \phi$ is an $m$-fold cover of $\R/\Z$ by a compact connected group, and this cover is isomorphic to the cover of $\R/\Z$ by itself using the multiplication map $x \mapsto mx$.  In other words, we have $\tilde \phi = m \tilde \phi'$ for some bijective character $\tilde \phi': G/H \to \R/\Z$, which can be lifted back to $\phi = m \phi'$ where $\phi'$ is the lift of $\tilde \phi'$.

Consider the function $g: J \to \R/\Z$ defined by $g(s) \coloneqq (\tilde \phi')^{-1}( \psi(s) )$.  Since $\psi(J)$ is compact, $g$ is continuous; since $\psi(s)$ lies in $\phi^{-1}(s)$, we have $m g(s) = s$ for all $s \in J$.  Also, $g(0)=0$.  By monodromy, this implies that 
$$g(s \text{ mod } \Z) = \frac{s}{m} \text{ mod } \Z$$
for all $s \in [a-b+\kappa', b-a+\kappa']$.  Since
$$ (C''-C'') \cap \phi^{-1}(J) = \bigcup_{s \in J} (\phi')^{-1}(g(s))$$
we conclude that $(C''-C'') \cap \phi^{-1}(J)$ is the Bohr set
\begin{equation}\label{ccp}
 (C''-C'') \cap \phi^{-1}(J) = (\phi')^{-1} (m^{-1} J)
\end{equation}
where $m^{-1} J \coloneqq [m^{-1}(a-b+\kappa'),m^{-1}(b-a-\kappa')] \text{ mod } \Z$. 

Having controlled $C''-C''$, we now return to $C''$.  We first need to relate $m$ with $\tau$.  On the one hand, for any $x \in C''$, we have 
$$ (C'' \cap \phi^{-1}( J + \phi(x) )) - x \subset (C''-C'') \cap \phi^{-1}(J \cap ([a,b] + \phi(x))).$$
From \eqref{fci}, the left-hand side has measure $\tau(b-a) + o(1)$ (we now allow the $o(1)$ terms to depend on $\kappa'$).  From \eqref{ccp}, the right-hand side has measure $m^{-1}(b-a)+o(1)$.  We conclude that
$$ \tau \leq m^{-1} + o(1).$$
On the other hand, from Lemma \ref{lodo} and \eqref{ccp} we see that
$$ \int_G 1_{C''} * 1_{-C''}\ d\mu_G \geq \int_{(\phi')^{-1} (m^{-1} J)}\tau (b-a-\|\phi(x)\|_{\R/\Z})\  d\mu_G(x) - o(1).$$
By \eqref{mugc-new}, the left-hand side is
$$ \mu_G(C'')^2 = \tau^2 (b-a)^2 + o(1).$$
By change of variables, the right-hand side is equal to
$$ \int_{m^{-1} J} \tau(b-a-\|ms\|_{\R/\Z})\ d\mu_{\R/\Z}(s) - o(1) = \tau m^{-1} (b-a)^2 - o(1)$$
and we conclude that
$$ \tau \geq m^{-1} + o(1).$$
Thus we have $\tau = m^{-1} + o(1)$; from \eqref{mugc-new} we conclude that
\begin{equation}\label{muhaha}
 \mu_G(C'') = m^{-1}(b-a) + o(1).
\end{equation}
From \eqref{fci}, there exists $x \in C''$ such that $\phi(x) = a+o(1)$.  Since $C''-x$ is contained in $(C''-C'') \cap [a-\phi(x),b-\phi(x)]$, it lies in $(C''-C'') \cap \phi^{-1}(J)$ outside of a set of measure $o(1)$.  Applying \eqref{ccp} and translating, we conclude that outside of a set of measure $o(1)$, $C''$ lies in $(\phi')^{-1}( m^{-1} J ) + x$; it also lies in $\phi^{-1}(I)$.  Thus, outside of a set of measure $o(1)$, $C''$ lies in $(\phi')^{-1}(J')$, where
$$ J' := \{ s \in \R/\Z: s \in m^{-1} J  + \phi'(x); ms \in I \}.$$
The set $\{ s \in \R/\Z: ms \in I \}$ is the union of $m$ equally spaced arcs of length $\frac{b-a}{m}$ each, while $m^{-1} J + \phi'(x)$ is an arc of length $2\frac{b-a}{m}$.  Since $b-a \leq \frac{1}{5}$, we conclude that $J'$ is an arc of length at most $\frac{b-a}{m}$; in particular, 
$$ \mu_G( (\phi')^{-1}(J') ) = m^{-1} (b-a).$$
Comparing this with \eqref{muhaha} we conclude that $C''$ is asymptotically equivalent to the Bohr set $(\phi')^{-1}(J')$, and hence $C$ is also, giving Theorem \ref{inv-4} (and thus Theorems \ref{inv-3}, \ref{inv-2}, and \ref{inv-1}).

\section{Further remarks}

It is natural to ask whether Theorem \ref{inv-1} or Theorem \ref{inv-2} may be extended to more general groups.  John Griesmer (personal communication) has proposed the following strong conjecture:

\begin{conjecture}\label{jg} Let $G$ be a compact group (not necessarily abelian) with probability Haar measure $\mu_G$, let $\eps > 0$, and let $\delta>0$ be sufficiently small depending on $\eps$.  Then for any compact subsets $A,B \subset G$ with $\mu_G(AB) \leq \mu_G(A)+\mu_G(B)+\delta$ and $\mu_G(A) + \mu_G(B) \leq 1-\eps$, there exists compact subsets $A',B'$ of $G$ with $\mu_G(A \Delta A'), \mu_G(B \Delta B') \leq \eps$ such that $\mu_G(A'B') \leq \mu_G(A') + \mu_G(B')$.
\end{conjecture}

One could strengthen this conjecture even further by requiring that $\delta$ be independent of $G$.  One can also consider non-compact groups $G$ (in which one would remove the hypothesis $\mu_G(A) + \mu_G(B) \leq 1-\eps$), though for non-unimodular groups there may be additional technical difficulties arising from the distinction between left-invariant and right-invariant Haar measures.  The case $A=B$ would be of particular interest, as it basically is concerned with classification of sets of doubling constant slightly larger than $2$.

Note that Theorem \ref{inv-1} verifies Conjecture \ref{jg} (with $\delta$ independent of $G$) under the additional hypotheses that $G$ is connected and abelian.  The case $G = \Z/p\Z$ of a cyclic group of prime order also follows from \cite[Theorem 21.8]{g}.  This conjecture would combine well with the extensive literature \cite{kemp2}, \cite{kneser}, \cite{g}, \cite{gri}, \cite{dev}, \cite{bjork}, \cite{bjork2}, \cite{bjork3} on classifying pairs of sets $A',B'$ obeying the relation $\mu_G(A'B') \leq \mu_G(A') + \mu_G(B')$ for various types of groups $G$.

It may also be possible to obtain an inverse theorem for Theorem \ref{ruzsa-thm}, that is to say to obtain some approximate structural description of sets $A,B$ for which
$$ \int_G \min( 1_A * 1_B, t )\ d\mu_G \leq t \min( \mu_G(A) + \mu_G(B) -t, 1 ) + \eps$$
for some $t>0$ and some small $\eps>0$, assuming appropriate non-degeneracy conditions on $\mu_G(A), \mu_G(B), t$.  We do not pursue this question here.

\end{document}